\documentclass[12pt]{amsart}
\usepackage[utf8]{inputenc}
\usepackage{Baskervaldx}
\usepackage[cal=cm]{mathalfa}
\usepackage{enumerate, geometry, graphicx, subcaption, amsmath, array, pdfsync, tikz, tikz-cd, color, url, float, wrapfig, comment}

\usepackage[baskervaldx,cmintegrals,bigdelims,vvarbb]{newtxmath}
\usepackage{mathrsfs,amsthm}
\usepackage{color}
\usepackage{fullpage}
\usepackage[all]{xy}
\renewcommand{\P}{{\mathbb{P}}}

\DeclareMathOperator{\CB}{CB}
\DeclareMathOperator{\irr}{irr}

\def\Sc{{\mathcal{S}}}
\def\Oc{{\mathcal{O}}}
\def\cl{{\colon}}
\def\ra{{\rightarrow}}
\def\dra{{\dashrightarrow}}
\def\Xbar{{\overline{X}}}
\def\BVA{{\mathrm{BVA}}}
\def\mfd{{\mathrm{mfd}}}

\def\gon{{\mathrm{gon}}}
\def\Hilb{{\mathrm{Hilb}}}
\def\Big{{\mathrm{Big}}}
\def\N{{\mathbb{N}}}
\def\R{{\mathbb{R}}}
\def\Z{{\mathbb{Z}}}
\def\Q{{\mathbb{Q}}}
\def\FC{{\mathrm{FC}}}
\def\MFC{{\mathrm{MFC}}}
\def\Mori{{\mathrm{Mori}}}
\def\NS{{\mathrm{NS}}}
\def\SL{{\mathrm{SL}}}

\title{Minimal degree fibrations in curves and the asymptotic degree of irrationality of divisors}

\theoremstyle{plain}
\newtheorem{theorem}{Theorem}
\newtheorem{proposition}[theorem]{Proposition}
\newtheorem{corollary}[theorem]{Corollary}

\newtheorem{lemma}[theorem]{Lemma}

\newtheorem{manualtheoreminner}{Theorem}
\newenvironment{manualtheorem}[1]{%
  \renewcommand\themanualtheoreminner{#1}%
  \manualtheoreminner
}{\endmanualtheoreminner}

\theoremstyle{definition}
\newtheorem{definition}[theorem]{Definition}
\newtheorem{example}[theorem]{Example}

\theoremstyle{remark}
\newtheorem{remark}[theorem]{Remark}

\newtheorem{question}[theorem]{Question}

\numberwithin{figure}{section}
\numberwithin{theorem}{section}
\numberwithin{equation}{section}

\author{Jake Levinson}
\address{
	Mathematics Department \\
	Simon Fraser University \\
	Burnaby, BC V5A 1S6 \\
	Canada}
\email{jake\_levinson@sfu.ca}

\author{David Stapleton }
\address{
	Mathematics Department \\
	University of Michigan \\
	Ann Arbor, MI 48109 \\
	USA}
\email{dajost@umich.edu}

\author{Brooke Ullery}
\address{
	Mathematics Department \\
	Emory University \\
	Atlanta, GA 30322 \\
	USA}
\email{bullery@emory.edu}

\thanks{During the preparation of this article the first author was partially supported by NSERC Discovery Grant RGPIN 2021-01469 and the second author was partially supported by NSF grant FRG-1952399.}

\begin{document}

\maketitle

\section*{Introduction}
\thispagestyle{empty}

The purpose of this paper is to relate the geometry of a projective variety $Y\subset \P^N$ to the degree of irrationality of divisors $X\subset Y$ of large degree.

Recall, for an $n$-dimensional variety $X$ the \textit{degree of irrationality of $X$} is
\[
\irr(X):=\min\{\delta>0 \mid \exists \phi\cl X\dra \P^n \text{ with }\deg(\phi)=\delta\}.
\]
Computing the degree of irrationality has recently attracted a great deal of interest (\cite{BCD14,BDELU,BY22,Chen21,SU20}). A feature of much of the previous work computing the degree of irrationality for divisors $X\subset Y$ has been that the low degree maps from $X$ to projective space factor through fibrations of $Y$ in curves,
\[
\begin{tikzcd}
X \arrow[dashed,d,swap,"\phi"] \arrow[hook,r] & Y \arrow[dashed, dl, "\psi"] \\ \mathbb{P}^n
\end{tikzcd}
\]
where $\phi$ is generically finite and $\psi$ is rationally fibered in curves. The prototype of such a result is due to M. Noether \cite{Noether1882} (Hartshorne \cite{HartshornedoesNoether} gave a modern treatment) who proved that any gonality map of a smooth, degree $d\ge 2$ plane curve is given by projection from a point -- and thus the fibers lie on lines in $\P^2$. In this paper we show this principle holds in a fairly general setting. As a consequence, we give strong bounds on $\irr(X)$ in terms of a natural invariant of the polarized variety $Y$.

Specifically, for a polarized variety $(Y,H)$ of dimension $n+1$ we associate a notion of degree to maps with positive dimensional fibers. If $\phi\cl Y\dra \P^n$ we define the \textit{$H$-degree} of $\phi$ as
\[
\deg_H(\phi) := \deg(H\cdot[C_\phi])
\]
where $C_\phi$ is the closure of a general fiber of $\phi$. Similarly to the degree of irrationality, for any pair we define the \textit{minimal fibering degree of $(Y,H)$} to be:
\[
\mfd(Y,H) := \min\{e>0 \mid \exists \phi\cl Y\dra \P^{n} \text{ with }\deg_H(\phi)=e\},
\]
i.e. it is the minimal degree $e$ such that $Y$ is rationally fibered in degree $e$ curves over $\P^n$.

\begin{manualtheorem}{A}\label{thmA} Let $Y$ be a smooth projective variety of dimension $n+1$ with a fixed ample divisor $A$. Fix an effective divisor $E$ on $Y$. Assume $d$ is sufficiently large (depending on $Y$, $A$, and $E$).
\begin{enumerate}
\item If $X\in |dA+E|$ is any smooth divisor and $\phi\cl X\dra \P^n$ is a rational map such that $\deg(\phi)= \irr(X)$ then there is a minimal $A$-degree fibration $\psi\cl Y\dra \P^n$ such that $\phi$ factors through $\psi$.
\item As a consequence:
\[
d\cdot \mfd(Y,A) - O(1) \le \irr(X) \le d\cdot \mfd(Y,A)+O(1).
\]
(If $E = 0$, then $\irr(X)\le d\cdot \mfd(Y,A)$.)
\item If every minimal degree fibration of $(Y,A)$ is regular (i.e. it is equivalent to a regular map after possible postcomposing with a Cremona transformation), then
\[
\irr(X) = \mfd(Y,X).
\]
\end{enumerate}
\end{manualtheorem}

It follows quickly that the degree of irrationality of a general complete intersection of hypersurfaces of unbalanced degrees $0\ll d_1\ll d_2\ll \cdots \ll d_r$ grows like the product of the degrees $d_1 \cdots d_r$.

\begin{manualtheorem}{B}\label{thmB}
Let $\epsilon >0$. Suppose $0\ll d_1\ll d_2\ll\dots \ll d_r$ are sufficiently unbalanced (depending on $\epsilon$) and let $X\subset \P^{N}$ be a general complete intersection of hypersurfaces of type $(d_1,\cdots,d_r)$, then
\[
(1-\epsilon)d_1\cdots d_r \le \irr(X)\le d_1\cdots d_r.
\]
\end{manualtheorem}

This gives an answer to a problem posed by Bastianelli et al. \cite[Prob. 4.1]{BDELU}. Previously, the problem of computing the degree of irrationality of complete intersections has been studied by Chen (\cite{Chen21}) in the case of codimensions and dimensions 2, and the second two authors (\cite{SU20}) for $(2,d)$ complete intersections and $(3,d)$ complete intersection surfaces/threefolds. Chen was able to establish multiplicative bounds on the degree of irrationality -- and in fact the covering gonality -- by giving interesting bounds on the multipoint Seshadri constants of divisors and complete intersection threefolds. Theorem~\ref{thmB} also gives a new proof (in the case of general curves of unbalanced degrees) that the gonality of a complete intersection curve is roughly the degree of the curve (\cite{HCU20,LazLecs}).

Our results use a recent theorem of Banerjee \cite{Banerjee} who studies when sets of points satisfying the Cayley--Bacharach condition lie on low degree curves. Let $X$ be a variety with linear system $\Lambda$. Recall that a set of points
\[
\Gamma =\{P_1,\dots,P_\gamma\}\subset X
\]
satisfies the \textit{Cayley--Bacharach condition} with respect to a linear system $\Lambda$ if for every point $P\in \Sc$: a divisor $D\in \Lambda$ contains $S$ $\iff$ $D$ contains $S\setminus P$. In the case $\Gamma\subset \P^N$ and $\Lambda$ is the complete linear system of $\Oc(r)$ we say \textit{$\Gamma$ satisfies $\CB(r)$.} Banerjee proves:

\begin{theorem}[\cite{Banerjee}]\label{Banerjee}
Let $\Gamma\subset \P^N$ be a set of points in a projective space satisfying $\CB(r)$. There are non-decreasing functions $G(e)$ and $H(e)$ such that if $r\ge G(e)$ and $|\Gamma|\le (e+1)r-H(e)$ then $\Gamma$ lies on a curve of degree $e$.
\end{theorem}

\noindent This is related to recent work of Picoco, which originally inspired this project. Picoco asks:

\begin{question}\label{PicocoQ}
(\cite[Ques. 1.1]{Picoco}) If $\Gamma$ satisfies $\CB(r)$ and $|\Gamma|\le (e+1)r-(e^2-e-1)$, does $\Gamma$ lie on a degree $e$ curve in $\P^N$?
\end{question}

\noindent Here we view Banerjee's result as an asymptotic answer to Picoco's question. In other words, if the answer to Question~\ref{PicocoQ} is yes, then in Banerjee's theorem we may take $G(e) = 0$ and $H(e) = e^2-e-1$.

In \S1 we introduce the minimal fibering degrees and provide some basic properties and first examples. In \S2 we recall the Cayley--Bacharach condition and some basic results about equations defining curves. In \S3 we prove Theorem~\ref{thmA}. In \S4 we prove Theorem~\ref{thmB} and compute some low degree examples effectively using Picoco's work.\\

\noindent\textit{Acknowledgments.} We thank Ishan Banerjee for sharing their work in preparation with us, and for enlightening conversations about their result. We thank Nathan Chen, Joe Harris, Robert Lazarsfeld, and Alex Perry for helpful conversations and suggestions.

\section{Minimal Fibering Degrees}

In this section we introduce the \textit{minimal fibering degree} of a pair (denoted $\mfd(Y,H)$) and describe some basic properties. We compute the functions for some basic examples.

\begin{definition}
Let $(Y,H)$ be a pair of $Y$ a smooth, projective variety of dimension $n+1$ and $H$ an effective divisor on $Y$. Define the \textit{$H$-degree} of a map $\phi\cl Y\dra \P^m$ to be
\[
\deg_H(\phi) := \deg(H^{(n+1)-m}\cap Y_b)
\]
where $Y_b$ is the closure of a general fiber of $\phi$.
\end{definition}

\begin{example}
In the case $m=n+1$, the $H$-degree is the same as the degree of $\phi$. However, if $X\in |H|$ is a hypersurface, and
\[
\psi\cl Y\dra \P^n
\]
is a rational map, it can happen that $\deg_H(\psi) \ne \deg (\psi|_X)$. For example, if we consider projection from a point on a smooth degree $d$ hypersurface $X\subset \P^{n+1}$. Then $\deg_H(\psi) = d$, whereas $\deg(\psi|_X) = d-1.$
\end{example}

\begin{lemma}\label{projectbound}
Assume that $Y\subset \P^N$ is a projective variety and there is a map
\[
\phi\cl Y\dra \P^m
\]
of relative dimension $\ell$ and $H$-degree $e$. If $m<n+1$ then there is a map $\phi'\cl Y\dra \P^{m+1}$ with $\deg_H(\phi') = e$.
\end{lemma}

\begin{proof}
With $\phi\cl Y\dra \P^m$ as above let $Y_b$ denote the closure of a general fiber. Consider the induced map:
\[
Y\dra \P^N\times \P^m.
\]
Let $\P^N\dra \P^1$ be a generic linear projection and fix a birational isomorphism $\P^1\times \P^m\simeq_{\mathrm{bir}}\P^{m+1}$. The closure of the general fiber of the composition
\[
\begin{tikzcd}
Y \arrow[d,dashed]\arrow[bend left=15,drr,dashed,"\phi'"] &&\\
\P^N\times \P^m \arrow[r,dashed]& \P^1\times \P^m\arrow[r,"\simeq"] &\P^{m+1}
\end{tikzcd}
\]
is a hyperplane section of $Y_b$. Thus $\deg_H(\phi') = \deg_H(\phi).$
\end{proof}

\begin{definition}
For an $n+1$ dimensional pair $(Y,H)$ and any integer $0\le \ell\le n+1$, the \textit{minimal fibering degree in dimension $\ell$} is the number
\[
\mfd_\ell(Y,H):=\min\{\deg_H(\phi)>0 \mid \phi\cl Y\dra \P^{(n+1)-\ell}\}.
\]
We will largely work with $\mfd_1(Y,H)$ which we just call \textit{the minimal fibering degree of $(Y,H)$} and we denote it by $\mfd(Y,H)$.
\end{definition}

\begin{definition}
It will be convenient to talk about the set of \textit{fiber classes of $Y$}, is a subset of the Mori cone of curves in $Y$:
\[
\FC(Y) := \left\{ [C_\psi] \in \Mori(Y) \ \middle|  \begin{array}{l}
C_\psi\text{ is the closure of a general fiber}\\
\text{of a dominant map }\psi\cl Y\dra \P^n
\end{array} \right\}\subset \Mori(Y).
\]
For any effective divisor $H$, we have
\[
\mfd(Y,H) = \min\{ H\cdot [C_\psi] \mid  [C_\psi]\in \FC(Y)\}.
\]
Accordingly, we define the set of \emph{minimal-degree fiber classes of $Y$ with respect to $H$},
\[
\MFC(Y,H) := \left\{ [C_\psi] \in \FC(Y) \mid \mfd(Y,H) = H \cdot [C_\psi] \right\}
\]
\end{definition}

\begin{proposition}[Elementary Properties]\label{mfdelemprops} Let $(Y,H)$ be a pair as above.
\begin{enumerate}
\item $\mfd(Y,dH) = d \mfd(Y,H).$
\item $\mfd(Y,H)\ge 0$. (More generally, if $H\ge H'$ then $\mfd(Y,H) \ge \mfd(Y,H').$)
\item The minimal fibering degree extends to a continuous, piecewise linear function on the big cone: 
\[
\mfd\cl \Big(Y) \ra \R
\]
\item If $H_d = H+\frac{1}{d}E$ for some big $\Z$-divisor $H$ and some effective divisor $E$, then for $d\gg 0$ there is an inclusion
\[
\MFC(Y,H_d) \subset \MFC(Y,H).
\]
\item If there are effective divisors $H_1, \dots, H_n\subset Y$ such that $\bigcap_{i=1}^n\MFC(Y,H_i) \ne \emptyset$ then for all $\alpha_1,\dots,\alpha_n>0$:
\[
\MFC(Y,\alpha_1H_1+\cdots + \alpha_nH_n) = \bigcap\limits_{i=1}^n\MFC(Y,H_i).
\]
In particular, the function $\mfd(Y, -)$ is linear on the convex hull of $H_1, \ldots, H_n$, and for any $[C] \in \bigcap_{i=1}^n \MFC(Y, H_i)$,
\[
\mfd(Y, \alpha_1 H_1 + \cdots + \alpha_n H_n) = (\alpha_1 H_1 + \cdots + \alpha_n H_n) \cdot [C].
\]

\item If $H$ is very ample then:
\[
\mfd_0(Y,H) \le \mfd_1(Y,H)\le \cdots \le \mfd_{n+1}(Y,H).
\]
Moreover, $\irr(Y) = \mfd_0(Y,H)$ and $\mfd_{n+1}(Y,H)$ is the degree of $Y$ in $\P^N$.

\end{enumerate}
\end{proposition}

\begin{proof}
(1) is straightforward from the definitions. (2) follows from the observation that the fibers of any map $\psi\cl Y\dra \P^n$ cover $Y$.

To prove (3), first we can extend to the rational divisor classes by (1). By duality, any $\R$-divisor class $\alpha \in \Big(Y)$ induces a linear map
\[
\alpha\cl \FC(Y) \ra \R^{> 0}.
\]
Define $\mfd(\alpha)$ to be the minimal $\R$-value obtained by this function. Note that this is actually a minimum. The big cone of $Y$ is the interior of the pseudoeffective cone which is dual to the movable cone of $Y$. All the curve classes are integral classes inside the movable cone, and these are pointed, full-dimensional cones. So for any real number $\lambda$, there are finitely curve classes $[C_\psi]$ such that $\alpha([C_\psi]) \le \lambda$. Now if we vary the class of $\alpha$ continuously, the images of the fiber classes $\FC(Y)$ also vary continuously, and thus so does the minimum. Lastly, it is straightforward to show that for $\alpha = H$ (a big integral divisor), then $\mfd(\alpha) = \mfd(Y,H)$. To show it is piecewise linear, for each $[C_\psi]\in \FC(Y)$ there is a (possibly empty) subcone of $\Big(Y)$ where $\alpha\cdot [C_\psi ] =\mfd(\alpha)$ and on this subset, $\mfd$ is linear by duality. Lastly, $\Big(Y)$ is a (possibly countable) union of such subsets.

To prove (4), suppose that $[C_\psi]\in \FC(Y)$ is an $H_d$-minimal class that is not $H$-minimal. Then there is an $H$-minimal class $[C_{\psi'}]\in \FC(Y)$ which satisfies:
\[
H_d\cdot [C_{\psi'}]\ge H_d\cdot [C_\psi] \ge H\cdot [C_\psi] > H \cdot [C_{\psi'}].
\]
The strict inequality is of integers. This implies
\[
E \cdot [C_{\psi'}] \ge d.
\]
We know there are finitely many classes $[C_{\psi'}]\in \MFC(Y,H)$. Thus, once $d> E \cdot [C_{\psi'}]$ for all such minimal classes, we are done.

To prove (5), let
\[
[C_{\psi}]\in \bigcap\limits_{i=1}^n\MFC(Y,H_i)\text{ and let }[C_{\phi}] \in \MFC(Y,\alpha_1H_1+\cdots +\alpha_nH_n).
\]
Then:
\begin{align*}
(\alpha_1H_1+\cdots +\alpha_nH_n)\cdot[C_\phi] & \le (\alpha_1H_1+\cdots +\alpha_nH_n)\cdot [C_\psi]\\
& = \sum \alpha_i (H_i\cdot[C_\psi])\\
& \le \sum \alpha_i (H_i\cdot [C_\phi]) = (\alpha_1H_1+\cdots +\alpha_nH_n)\cdot[C_\phi].
\end{align*}
This implies all the degrees $H_i\cdot[C_\psi] = H_i\cdot[C_\phi]$, and the result follows.

In (6), the inequalities follow from Lemma~\ref{projectbound}, and the rest are straightforward.
\end{proof}

\begin{example}
Suppose $Y$ is a surface with Picard rank 1, i.e. $\NS(Y) = \Z\cdot H$. If $a$ is the smallest integer such that $aH$ moves in a pencil then $\mfd(Y,H) = a H^2$.
\end{example}

\begin{corollary}[Corollary of Thm.~\ref{thmA}]
Let $Y$ be as above. If $d$ is sufficiently large, and $X\in |dH|$ is smooth of genus $g$ then
\[
\gon(X) \approx a\cdot d \cdot H^2 \approx C \sqrt{g}.
\]
\end{corollary}

\begin{example}
Consider the product of two curves $Y = C_1\times C_2$ such that
\[
\NS(C_1\times C_2) \cong \Z f_1\oplus \Z f_2
\]
(where the $f_i$ represents the fiber of projection onto the $i$th factor). In other words, we are assuming the curves are general with respect to each other in that their Jacobians have no isogeny factors in common. Let $\delta_i$ be the gonality of $C_i$. By the K\"unneth formula, if a line bundle $L \equiv_{\mathrm{num}} a_1 f_1+a_2 f_2$ has no fixed curve and at least two sections then one of the following holds:
\begin{center}
(1) $a_1\ge \delta_1$ and $a_2\ge \delta_2$, (2) $a_1\ge \delta_1$ and $a_2 = 0$, or (3) $a_1 = 0$ and $a_2 \ge\delta_2$.
\end{center}
It follows quickly that the curve classes computing the minimal degree fibrations are either $\delta_1 f_1$ or $\delta_2 f_2$. As a consequence, in the big cone of curves (which is just the cone generated by $f_1$ and $f_2$), the function $\mfd$ is computed by intersecting against $\delta_1 f_1$ or $\delta_2 f_2$.
\end{example}

\begin{corollary}[Corollary of Thm.~\ref{thmA}]
Let $Y = C_1\times C_2$ as above, let $A$ be an ample divisor on $Y$, and let $E$ be an effective divisor on $Y$. Suppose $d$ is sufficiently large. If $X\in |dA + E|$ is smooth?, then any map $\phi\cl X\ra \P^1$ that computes the gonality of $X$ factors through a projection
\[
p_i\cl X\ra C_i
\]
followed by a map computing the gonality of $C_i$.
\end{corollary}

\begin{center}
\includegraphics[scale=.29]{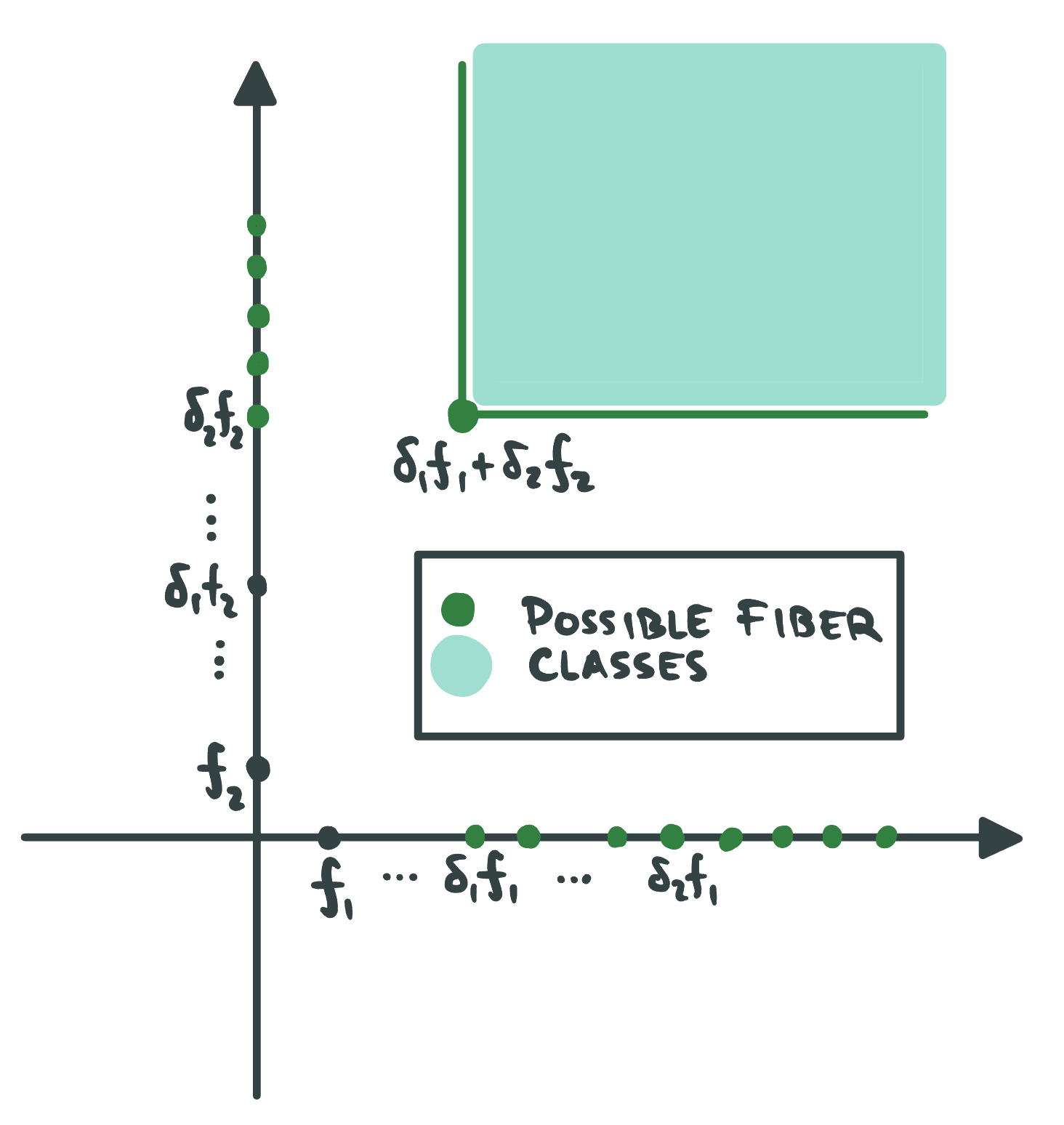}\quad\quad\includegraphics[scale=.29]{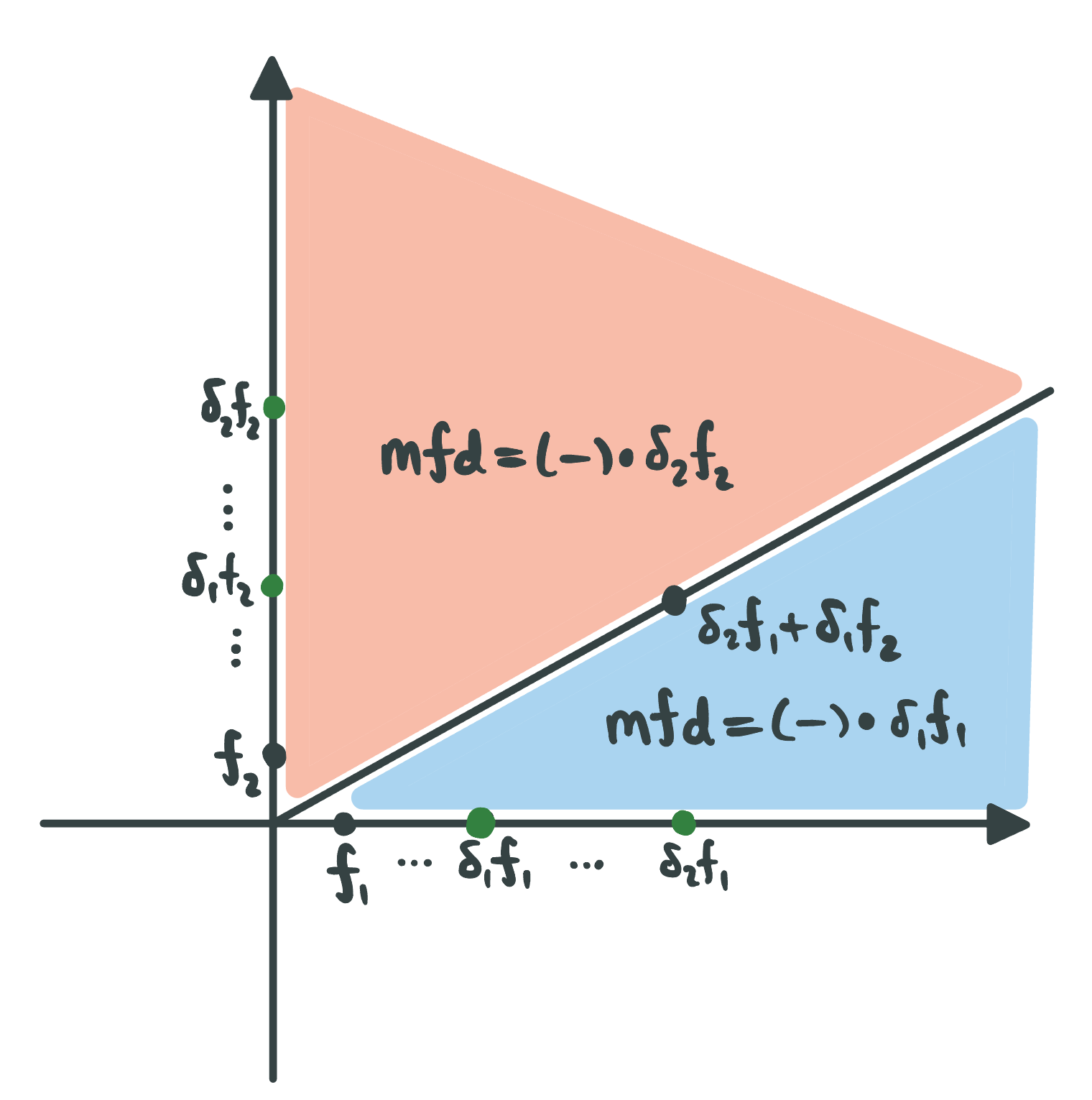}
\end{center}

\begin{example} Let $E$ be an elliptic curve without extra endomorphisms and consider the square $E\times E$. Then
\begin{center}
$\NS(E\times E) \cong \Z f_1 \oplus \Z f_2 \oplus \Z \Delta$
\end{center}
where $f_1$ and $f_2$ are the fibers of the projections as in the previous example and $\Delta$ is the class of the diagonal. These have the following intersection numbers:
\[
f_1^2=f_2^2=\Delta^2=0\text{ and }f_1\cdot \Delta=f_2\cdot\Delta=f_1\cdot f_2 =1.
\]
The nef cone and the pseudoeffective cone are equal in this example and are given by the classes $\alpha\in \NS(E\times E)$ such that $\alpha^2\ge 0$ and $\alpha\cdot (f_1+f_2+\Delta)\ge 0$. To express the function $\mfd$ on the big cone, we take a cross-section of the cone (which is a circle) and extend by homogeneity.

\begin{figure}[htb]
\centering
    \includegraphics[width=8cm]{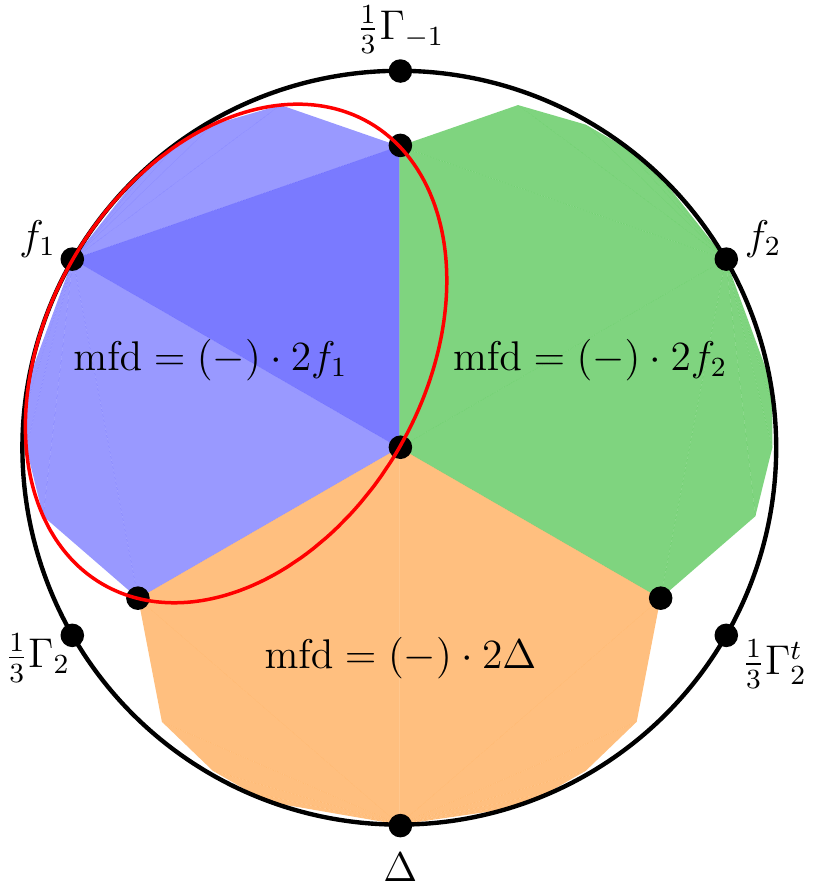}
        \caption{The function $\mfd$ on 
$\Big(E\times E)$ expressed as a piecewise linear function on the cross-section $(f_1+f_2+\Delta)\cdot (-) = 2$. The shaded triangle is a fundamental domain for $\SL_2(\Z)$.}
\end{figure}

Now, $2f_1$, $2f_2$ and $2\Delta$ all move in a pencil. We first find a region on which $\mfd(E \times E, H) = H \cdot 2f_1$. An explicit computation shows that
\begin{center}
$\mfd(E\times E,f_1+f_2+\Delta) = 4$
\end{center}
and for any
\begin{center}
$\psi\cl E\times E\dra \P^1$
\end{center}
the fiber $[C_\psi]$ is numerically equivalent to one of $2f_1$, $2f_2$, or $2\Delta$. Writing $\Gamma_n \subset E \times E$ for the the graph of multiplication by $n$ on $E$, we likewise have
\begin{center}
$\mfd(E\times E, f_1+f_2+\Gamma_{-1}) = 4$
\end{center}
\noindent and is computed by the pencils numerically equivalent to $2f_1$, $2 f_2$, or $2\Gamma_{-1}$ and finally
\begin{center}
$\mfd(E\times E, f_1) = 0$
\end{center}
and is computed by any pencil with fiber class numerically equivalent to a multiple $a f_1$ with $a\ge 2$. By Prop.~\ref{mfdelemprops}(5), we see that for any class $H$ that is a positive linear combination of $f_1+f_2+\Delta$, $f_1+f_2+\Gamma_{-1}$, and $f_1$,  \begin{center}
$\mfd(E\times E, H) = H\cdot 2f_1$
\end{center}
Finally, we can compute the rest of the function $\mfd$ on $\Big(E\times E)$ using the action of $\SL_2(\Z)$ on $E\times E$. The result is pictured in Figure 1.1, with the convex hull of $f_1$, $f_1+f_2+\Delta$, $f_1+f_2 + \Gamma_{-1}$ depicted as a shaded triangular region. For example, the action of the subgroup
\begin{center}
$\left\{\begin{bmatrix}
1&0\\
a&1
\end{bmatrix} \ \middle|\ a\in \Z\right\} \le \SL_2(\Z)$
\end{center}
fixes the class of $f_1$. Thus for any point in the union of the orbit of the blue triangle under this subgroup, we have $\mfd = (-)\cdot 2f_1$. We note there is a connection to the order-3 apeirogonal tiling of the circle.
\end{example}

\begin{corollary}[Corollary of Thm.\ref{thmA}]
Let $Y = E\times E$ as above. Let $A$ be an ample divisor on $Y$, let $E$ be an effective divisor on $Y$, and suppose $d$ is sufficiently large. If $X\in |dA + E|$ is smooth, then any map $\phi\cl X\ra \P^1$ that computes the gonality of $X$ factors through a fibration
\[
p_i\cl E\times E\ra E
\]
followed by a 2:1 map $E \ra \P^1$.
\end{corollary}

\begin{question}
The minimal fibering degree seems to be a natural geometric invariant, and it would be very interesting to understand its behavior for natural classes of varieties. Here we list a few examples of interest.
\begin{enumerate}
\item If $(A,\Theta)$ is a principally polarized abelian variety with $\NS(A) = \Z\cdot \Theta$, can we compute $\mfd(A,\Theta)$?
\item If $A$ is an abelian variety that is isogenous to a product of elliptic curves, is every minimal degree fibration equivalent to a regular morphism (after postcomposing with a Cremona transformation)?
\item If $Y$ is a toric variety, does the function $\mfd$ have a nice combinatorial description?
\item If $Y$ is a \textit{very general} hypersurface (or a complete intersection) of dimension $n+1\ge 3$ and degree $a$, then it seems reasonable to expect that $\mfd(Y,H) = a$ (we give some results in this direction for general complete intersections in \S4). This can be deduced in some cases by work on the failure of the integral Hodge conjecture for these varieties (see \cite{Paulsen}), but it seems possible that it could be easier to prove these results directly.
\end{enumerate}
\end{question}

\section{Elementary properties of equations defining curves and the Cayley--Bacharach condition}

Now we prove some results about curves to set up the proof of Theorem~\ref{thmA}. The following lemma is well known.

\begin{lemma}\label{defeqns}
A degree $e$ curve $C\subset \P^N$ is cut out set-theoretically by equations of degree at most $e$.
\end{lemma}

\begin{proof}
Let $p\in \P^N$ be a point with $p\not\in C$. For a general projection $\pi\cl \P^N \dra \P^2$ the image of $\pi(C)$ and $\pi(p)$ do not intersect. The closure of $\pi^{-1}(\pi(C))$ is a degree $e$ hypersurface containing $C$ that avoids $p$.
\end{proof}

\begin{corollary}\label{samecurve}
Let $C_1$ and $C_2$ be two curves in $\P^N$ with degrees $e_1$ and $e_2$. If $C_1$ is irreducible and $\#(C_1\cap C_2) \ge e_1 e_2+1$, then $C_1\subset C_2$.
\end{corollary}

\begin{proof}
Let $D\subset \P^N$ be a hypersurface of degree at most $e_2$ that contains $C_2$. Then $|C_1\cap D| \ge \#(C_1 \cap C_2) \ge e_1e_2+1$. By Bezout's theorem, as $C_1$ is irreducible, $C_1\subset D$. By Lemma \ref{defeqns}, $C_2$ is cut out set theoretically by hypersurfaces of degree at most $e_2$, so by intersecting all such hypersurfaces we see $C_1\subset C_2$.
\end{proof}

We restate the Cayley--Bacharach condition from the introduction.

\begin{definition}
Let $\Gamma\subset X$ be a finite set of points and let $\Lambda$ be a linear system on $X$. A set of points
\[
\Gamma =\{P_1,\dots,P_\gamma\}\subset X
\]
satisfies the \textit{Cayley--Bacharach condition} with respect to $\Lambda$ if for every point $P\in \Sc$: a divisor $D\in \Lambda$ contains $S$ $\iff$ $D$ contains $S\setminus P$. In the special case that $X=\P^N$ and $\Lambda = |\Oc(r)|$ we say $\Gamma$ satisfies $\CB(r)$.
\end{definition}

The following lemma is elementary.

\begin{lemma}\label{elemprops}
Suppose $\Gamma\subset \P^N$ satisfies $\CB(r)$.
\begin{enumerate}
\item{\cite[Lem 2.4]{BCD}} If $\Gamma\ne \emptyset$ then $|\Gamma|\ge r+2$.
\item If $D\subset\P^N$ is a hypersurface of degree $d$, then $\Gamma\setminus (\Gamma\cap D)$ satisfies $\CB(r-d)$.
\end{enumerate}
\end{lemma}

\begin{proof}
The proof of (2) is straightforward.
\end{proof}

\begin{lemma}\label{ptspercomp}
Let $C = C_1\cup \cdots C_\ell\subset \P^N$ be a curve of degree $f$, with irreducible components $C_i$ of degrees $f_i$. Suppose that $\Gamma\subset C$ is a set of points that satisfy $\CB(r)$. Let $\Gamma_i$ be the set of points in $\Gamma$ that are contained in $C_i$ and not contained in any other curve. Then either $\Gamma_i$ is empty or $\Gamma_i$ satisfies $\CB(r+f_i-f)$. In particular, if $\Gamma_i$ is not empty then
\[
|\Gamma_i|\ge r+f_i-f+2.
\]
\end{lemma}

\begin{proof}
Assume $\Gamma_i\ne \emptyset$. Define $C^*_i := \bigcup_{j\ne i} C_j$ and $\Gamma^*_i:=\Gamma\setminus\Gamma_i$. Then $C_i^*$ has degree $f-f_i$. By Lemma~\ref{defeqns}, $C_i^*$ is cut out (set theoretically) by equations of degree $f-f_i$. Thus there is a hypersurface $D\subset \P^N$ of degree $f-f_i$ that contains $C^*_i$ and avoids $\Gamma_i$. By Lemma~\ref{elemprops}, $\Gamma_i$ satisfies $\CB(r+f_i-f)$ and so by \cite[Lem. 2.4]{BCD}:
\[
|\Gamma_i|\ge r+f_i-f+2.\qedhere
\]
\end{proof}

\section{The asymptotic degree of irrationality of high degree hypersurfaces}

This section is the core of the paper and contains the proof of Theorem~\ref{thmA}. We show that maps computing the degree of irrationality of large degree hypersurfaces on a projective variety factor through minimal degree fibrations of that variety in curves. Throughout this section we use Banerjee's result (Thm.~\ref{Banerjee}). 

Our proof of Theorem~\ref{thmA} proceeds as follows. We reduce to the case that $X\in |dH+E|$ for $H$ the restriction of a hyperplane class of an embedding $Y\subset \P^N$. We consider fibers $\Gamma \subset X$ of a minimal-degree map $\phi : X \dra \P^n$. By taking $d\gg 0$, we successively show:
\begin{itemize}
    \item $\Gamma$ lies on a curve $C$ of degree $f \leq \mfd(Y, H)$ (Lemma \ref{degreebound});
    \item $C$ is unique and lies on $Y$ (Lemma \ref{contY});
    \item these curves $C$ sweep out $Y$ birationally (completing the proof of Theorem A).
\end{itemize}
The necessary bounds on $d$ come from certain auxiliary constants $\zeta$ (Eq. \eqref{eqn:zeta}), $\ell$ (Eq. \eqref{eqn:ell}) and $M$ (Eq. \eqref{eqn:M}) defined below, depending only on $Y, H, E$.
 
To start we briefly recall an important notion of birational positivity, referred to as $\BVA_p$, which stands for \emph{birationally $p$-very ample}.

\begin{definition}
Recall that a line bundle $L$ on a projective variety $X$ \textit{satisfies property} $\BVA_p$ if there is a nonempty open subset $U\subset X$ such that for all finite subschemes $\xi\subset U$ of length $p+1$:
\[
H^0(X,L)\ra H^0(X,L|_\xi)
\]
is surjective.
\end{definition}

\begin{remark}
The connection to our work here is the following. If $\omega_X$ satisfies $\BVA_p$ then $X$ is not swept out by curves of gonality $\le p+1$ (\cite[Thm. 2.3]{BDELU}), so $\irr(X)\ge p+2$.
\end{remark}

\noindent\textbf{Assumptions.} Let $Y\subset \P^N$ be a smooth subvariety of dimension $n+1$. Let $H$ be the hyperplane class restricted to $Y$ and let $E$ be any effective divisor on $Y$. Let $X\in |\Oc_Y(d H + E)|$ be a smooth divisor. Let $\phi\cl X\dra \P^n$ be a rational map with $\deg(\phi) = \irr(X)$ and let $\Gamma\subset X$ be a general fiber of $\phi$. Let $e:= \mfd(Y,H)$.\\

We now analyze the projective geometry of $\Gamma\subset \P^N$, and study the minimal degree curves in $\P^N$ that interpolate $\Gamma$. Define:
\begin{equation}\label{eqn:zeta}
\zeta:=\min\{ m \in \mathbb{Z} \mid  \omega_Y(mH+E) \text{ is globally generated}\}.
\end{equation}
Geometrically, when $d \geq \zeta$, the canonical bundle of $X$ becomes globally generated, as follows:
\begin{lemma}\label{CBGamma}
The line bundle $\omega_X$ satisfies $\BVA_{d-\zeta}$ and thus $\Gamma$ satisfies $\CB(d-\zeta)$.
\end{lemma}

\begin{proof}
By adjunction
\[
\omega_X \cong \omega_Y(dH+E)|_X\cong \omega_Y(\zeta H+E)|_X\otimes \Oc_X((d-\zeta)H).
\]
Now $\omega_Y(\zeta H + E)|_X$ is effective. (By definition of $\zeta$, $\omega_Y(\zeta H+E)$ is globally generated. Thus we can take a section that does not contain $X$.) $\Gamma$ satisfies $\CB(d-\zeta)$ by \cite[Ex. 1.2(i),(iv)]{BDELU} and \cite[Thm. 2.3]{BCD}.

\end{proof}

Set
\begin{equation}\label{eqn:ell}
\ell := \min\{ m\in \N \mid mH \ge E\}.
\end{equation}

\begin{lemma}\label{degreebound}\hspace{1pt}
\begin{enumerate}
\item If $d\ge \zeta+e-1$ then there is an upper bound
\[
\irr(X) \le d\cdot e + O(1).
\]
\item If $E$ is a $\Q$-multiple of $H$ then $\irr(X) \le \mfd(Y,X)$.
\item If $d\gg 0$ then $\Gamma$ lies on a curve of degree at most $e$ in $\P^N$.
\end{enumerate}
\end{lemma}

\begin{proof}
Consider a map $\psi\cl Y\dra\P^n$ such that $e=\deg_H(\psi)$. As $Y$ is smooth this can be defined away from codimension 2, so we consider the map $\psi|_X$.

If $\psi|_X$ is dominant, then
\[
\irr(X)\le \deg(\psi|_X) \le \deg_{dH+E}(\psi).
\]
If $E$ is numerically a $\Q$-multiple of $H$, then $\deg_{dH+E}(\psi) = \mfd(Y,X)$. Otherwise, the right hand side is bounded by
\[
\deg_{(d+\ell)H}(\psi) = (d+\ell) \cdot e.
\]
Thus
\[
\irr(X) \le d\cdot e + O(1)
\]
where we can take the constant term to be $\ell \cdot e.$

Assume for contradiction that $\psi|_X$ is not dominant. Consider the rational map
\[
\P^n \dra \mathrm{Hilb}_{e}
\]
(the Hilbert scheme of degree $e$ curves in $\P^N$) induced by $\psi$. Let $B \to \mathrm{Hilb}_e$ be a resolution of this map (by abuse of notation, the map $\psi$ can be thought of as having codomain $B$). Let $\pi : C_B \to B$ be the pullback of the universal family of $\mathrm{Hilb}_e$ to $B$. So we have a diagram:
\[
\begin{tikzcd}
X\arrow[hook,d]\arrow[bend right=80,dashed,ddr,swap,"\psi|_X"]& \Xbar \arrow[l]\arrow[hook,d] \\
Y\arrow[dashed,dr,swap,"\psi"] & C_B \arrow[d,"\pi"]\arrow[l,"p"]\\
 & B.
\end{tikzcd}
\]
The map $C_B \to Y$ is birational. Define $\overline{X} \to X$ to be the strict transform of $X$ in $C_B$ and $Z$ the image of $\Xbar$ in $B$. We have $\Xbar \subseteq \pi^{-1}(Z)$. This shows that if $X$ is contracted by $\psi$, then it is swept out by curves of degree $\le e$. However, a curve of degree $\le e$ has gonality $\le e$, whereas by Lemma~\ref{CBGamma} and \cite[Lem. 1.3]{BDELU}, $X$ has covering gonality at least $d-\zeta+2$. This gives a contradiction once $d\ge \zeta+e-1$.

Now let $G(e)$ and $H(e)$ be the functions from Theorem~\ref{Banerjee}. As $d\gg0$, $d-\zeta \ge G(e)$. Likewise, as $d\gg0$ we have
\[
|\Gamma|\le ed+ e\ell \le (e+1)(d-\zeta) -H(e).
\]
Therefore by Theorem~\ref{Banerjee}, we have $\Gamma$ lies on a curve of degree $\le e$.
\end{proof}

\begin{remark}
In particular, if $G(e)$ and $H(e)$ are the functions from Banerjee's Theorem, as soon as
\[
d\ge \max\{G(e)+\zeta, e\ell +(e+1)\zeta+H(e)\}
\]
then $\Gamma$ lies on a curve of degree at most $e$.
\end{remark}

Let $C\subset \P^N$ be a curve of minimal degree through the general fiber $\Gamma$ of $\phi$, and set $f = \deg(C)$. Let $C_1, \ldots, C_\ell$ be the irreducible components of $C$, and set $f_i=\deg(C_i)$. By minimality, each $C_i$ contains a point in $\Gamma$ that is not contained in any other component.

We now show that for $d \gg 0$ (depending on $Y$), $C$ is unique and $C \subset Y$. Define:
\begin{equation}\label{eqn:M}
M:=\min\{ m \mid Y \text{ is the set-theoretic intersection of degree}\le m\text{ hypersurfaces}\}.
\end{equation}
\begin{lemma}\label{contY} Assume that $f\le e$ (which holds if $d\gg 0$) and assume that $d\ge e+\zeta+ M\cdot e-1$.
\begin{enumerate}
\item Every component $C_i\subset C$ contains at least $M\cdot f_i+1$ points that are not contained in any other component, and it follows that $C\subset Y$.
\item Assume further that $d\ge e^2+e+\zeta-1$. Then every component $C_i\subset C$ contains at least $e^2+1$ points of $\Gamma$ not contained in any other component. Moreover, there is a unique curve $C$ of minimal degree through $\Gamma$.
\end{enumerate}
\end{lemma}

\begin{proof}
First we prove (1). Let $\Gamma_i$ be the set of points of $\Gamma$ contained in $C_i$ and not contained in any other component of $C$. By Lemma~\ref{ptspercomp},
\[
|\Gamma_i|\ge d-\zeta+f_i-f+2.
\]
By our assumption on $d$ and the fact that $e\ge f\ge f_i$, we see
\[
|\Gamma_i| \ge  M\cdot f_i+1.
\]
Therefore, by Bezout's theorem $C_i$ is contained in $Y$.

To prove (2), by Lemma~\ref{ptspercomp}, Lemma~\ref{CBGamma}, and the assumption on $d$, each component $C_i$ contains at at least 
\[d-\zeta -e+2\ge e^2+1 (\ge f_i f + 1)\] points of $\Gamma$ not contained in any other component of $C$. 

Suppose $C'$ is any curve of minimal degree containing $\Gamma$. By assumption $\deg(C') = f$, so by Corollary~\ref{samecurve}, $C_i \subset C$, and thus $C \subset C'$. By symmetry, $C = C'$.
\end{proof}

We now complete the proof of Theorem A by showing that the minimal-degree curves $C$ containing the fibers $\Gamma$ sweep out $Y$ birationally.

\begin{proof}[Proof of Theorem~\ref{thmA}]

To start, assume $A$ is \textit{not} very ample. Let
\[
V:=\min\left\{ m>0\ \middle|\begin{array}{l}
mA\text{ is very ample, and the divisors}\\
(m+1)A,\dots,(2m-1)A\text{ are all effective}
\end{array}\right\}.
\]
Let $H=V\cdot A$. So, for any $d$ sufficiently large, if $q = \lfloor d/V\rfloor$ then
\[
|dA+E| \in \{ |q H+E|, |(q-1)H+((V+1)A+E)|,\dots,|(q-1)H+((2V-1)A+E)|\}.
\]
Now the minimal $A$-degree fibrations are the same as the minimal $H$-degree fibrations. So proving (1) for $d$ sufficiently large is equivalant to proving (1) for $q$ sufficiently large. So for the rest of the argument we assume that $A=H$ is very ample.

Consider the Hilbert scheme $\Hilb_f$ of degree $f$ ($\le e$) curves in $\P^N$ and with universal curve $C_f\ra \Hilb_f$. By Lemma~\ref{contY}(2) there is a map
\[
\P^n \dra \Hilb_f
\]
that sends a general point $b\in \P^n$ to the unique degree $f$ curve that contains the fiber $\Gamma=\phi^{-1}(b)$. We may assume that a general point $b\in\P^n$ maps to a reduced curve $[C]\in \Hilb_f$. Resolving the rational map gives a regular map:
\[
\P^n\simeq_{\mathrm{bir}} B\ra \Hilb_f.
\]
Consider the following diagram:
\[
\begin{tikzcd}
X\arrow[hook,d]\arrow[bend right=60,dashed,ddr,swap,"\phi"]& \Xbar \arrow[l]\arrow[hook,d] &\\
Y & C_B \arrow[d,"\pi"]\arrow[r]\arrow[l,"p"]& C_f\arrow[d]\\
 & B\arrow[r] & \Hilb_f.
\end{tikzcd}
\]
Here $C_B\subset \P^N\times B$ is the pullback of the universal curve $C_f$ to $B$. By Lemma~\ref{contY}(1) the projection to $\P^N$ lands inside $Y$; this gives the map $p$. There is a natural rational map $X\dra C_B$ given by sending a general point $x\in X$ to $(x,\phi(x))\in C_B\subset \P^N \times B$. Let $\Xbar$ be the closure of the image of $X$ in $C_B$.

Let $C'_B \subseteq C_B$ be any irreducible component. By flatness $C'_B$ dominates $B$ and has dimension $n+1$. We claim that $p : C'_B \to Y$ is likewise dominant, in particular generically finite. Let $b \in B$ be a general point and let $\Gamma_b \subseteq X$ and $C_b \subseteq Y$ be the corresponding fiber and curve. Let $C'_b = C'_B \cap C_b$. By Lemma~\ref{contY}(2), $C'_b$ contains a point of $\Gamma_b$ (in fact at least $e^2+1$ points). Thus $p(C'_B) \cap X \dra B$ is dominant, so since $X$ is irreducible, $p(C'_B) \supseteq X$. By Lemma~\ref{degreebound}, $X$ cannot be swept out by curves of degree $\leq e$, so $p(C'_B) \ne X$, from which we conclude $p(C'_B) = Y$.

Consider the cycle $[C_B] \in \mathrm{CH}_{n+1}(Y\times B)$. Let $p_1$ (resp. $p_2$) denote the projections of $Y\times B$ onto $Y$ (resp. $B$). To show that $C_B$ is birational to $Y$ it suffices to show
\[
(p_1)_* [C_B] = [Y].
\]
By the projection formula, this is the same as showing:
\[
(p_1)_*((p_1^*[X])\cdot [C_B]) = [X].
\]
Now,
\[
(p_1^*[X])\cdot [C_B] = a[\Xbar] + \sum a_i [E_i].
\]
For some $n$-dimensional subvarieties $E_i\subset C_B$ with $a,a_i>0$ (we know the $a_i>0$ as each component $C'_B$ dominates $Y$ and $X$ is a Cartier divisor). It suffices to show that $a=1$ and $(p_1)_*[E_i]=0$.

By the projection formula
\[
df+\ell f \ge df + [E]\cdot (p_2)*[C_b] = (p_1^*[X])\cdot [C_b].
\]
Expanding $p_1^*[X]$ gives:
\begin{align*}
(p_1^*[X])\cdot [C_b] &= a[\Xbar]\cdot [C_b]+\sum a_i [E_i]\cdot [C_b]\\
 &\ge a [\Xbar]\cdot [C_b]=a|\Gamma|.
\end{align*}
Let $G$ and $H$ be the functions from Banerjee's Theorem (Theorem~\ref{Banerjee}). As $d\gg 0$ we may assume $d-\zeta \ge G(f-1)$. Since $\Gamma$ satisfies $\CB(d-\zeta)$ and is not contained in a curve of degree $f-1$, we obtain
\begin{equation}\label{eqn:gammabound}
|\Gamma| \ge (d-\zeta)f-H(f-1). \end{equation}
Thus:
\[
df+\ell f\ge a|\Gamma|\ge a((d-\zeta)f-H(f-1)).
\]
If $d\gg 0$, this implies $a=1$.

It remains to show that $(p_1)_*[E_i]=0$ for each $[E_i]$. There are two cases: (1) the map $E_i\ra B$ is a contraction, or (2) the map $E_i\ra B$ is dominant.  In the case (1), if $p_1|_{E_i}$ is generically finite this implies that $X$ is swept out by curves of degree $\le e$ which contradicts Lemma~\ref{degreebound}.

In case (2), note that the general fiber of $E_i$ satisfies $\CB(d-\zeta)$ (by \cite[Thm. 2.3]{BCD}). Therefore, the degree of $E_i\ra B$ is at least $d-\zeta+2$. Now we redo the computation that showed $a=1$ again assuming assumption (1), this time accounting for $E_i$:
\begin{align*}
d f + \ell f\ge df+ (p_1^*[E])\cdot[C_b]&=(p_1^*[X])\cdot [C_b]\\
 &\ge [\Xbar]\cdot [C_b]+a_i [E_i]\cdot [C_b]\\
 &\ge |\Gamma|+d-\zeta+2\\
 &\ge (d-\zeta)f-H(f-1) + d-\zeta+2.
\end{align*}
This is impossible if $d\gg 0$.

Thus the map $\phi$ factors through a fibration $\psi\cl Y\dra \P^n$ and it follows that $f=e$. This completes the proof of (1).

To prove (2), we again reduce to the case $A$ very ample, and the lower bound on $\irr(X)$ then comes out of Eq.~\eqref{eqn:gammabound} (setting $f=e$). To prove (3), we also reduce to the case $A$ is very ample and in this case the effective parts are all $\Q$-multiples of $A$ so we are done by Lemma~\ref{degreebound}. The proof of (4) is just the observation that when $\psi$ is a regular morphism:
\[
\deg_{[X]}(\psi) = \deg(\psi|_{X}).\qedhere
\]
\end{proof}

\begin{remark}
When the divisor $A$ is very ample, we can make the lower bound on $d$ explicit in terms of the functions $G(e)$ and $H(e)$ in Theorem~\ref{Banerjee} and the other constants introduced in the section. Let

\[
d_0:=\max\left\{\begin{array}{rl}
G(e)+\zeta,(e+1)\zeta+e\ell+H(e), \zeta+e-1&(\text{Lem.~\ref{degreebound}})\\
e+\zeta+M\cdot e -1, e^2+e+\zeta-1,&(\text{Lem.~\ref{contY}})\\
2(\ell+\zeta+H(e-1))&(\text{proof of Thm.~\ref{thmA}})\\
2\zeta+\ell+H(e-1)-1&\\
(e+1)\zeta+e\ell+H(e-1)-1&
\end{array} \right\}.
\]
If $d\ge d_0$, then the map $\phi\cl X\dra \P^n$ factors through a minimal degree fibration of $Y$ and
\[
\irr(X) > d\cdot e-\zeta\cdot e -H(e-1).
\]
\end{remark}

\begin{remark}\label{PicocoQRemark}
If we further assume the answer to Question~\ref{PicocoQ} is yes for degrees up to $e$ then we can make these bounds more concrete. In this case $G(e)=0$ and $H(e)=e^2-e-1$. Set
\[
d'_0:=\max\left\{ \zeta+M\cdot e-1, (e+1)(\zeta+\ell)+2(e^2-e-1), \zeta+(e+3)\ell + 2e^2\right\}.
\]
If $d>d'_0$ then any map computing $\irr(X)$ factors through a minimal degree fibration of $Y$ and $\irr(X)>d\cdot e-\zeta\cdot e -e^2+e+1.$
\end{remark}

\section{The degree of irrationality of unbalanced complete intersections}

In this section we apply Theorem~\ref{thmA} to compute the degree of irrationality in various examples. To start we consider complete intersections of unbalanced degrees. Here we are only concerned with the asymptotic results, so we do not attempt to optimize the stated inequalities.

\begin{proof}[Proof of Theorem B]
We proceed by induction on the codimension. By \cite[Thm.~A]{BDELU}, for hypersurfaces $X_{d_1}\subset \P^{N}$ it is known that
\[
d_1\ge \irr(X)\ge d_1-N+1.
\]
Let $\epsilon >0$, if $d_1$ is large enough then the right hand side is at least $(1-\epsilon)d_1$. (Alternatively, we can apply Theorem~\ref{thmA} to $\P^N$ using that $\mfd_1(\P^N) = 1$.)

Now assume $\epsilon >0$ is small and define $\epsilon'>0$ by $1-\epsilon' = \sqrt{1-\epsilon}$. Suppose $d_1, \ldots, d_{r-1}$ are such that a general complete intersection $Y$ of these degrees has
h
\[
\irr(Y)\ge (1-\epsilon')d_1\cdots d_{r-1}.
\]
Then by Lemma~\ref{mfdelemprops}(6) we have
\[
d_1\cdots d_{r-1} = \mfd_{N-r+1}(Y,H)\ge \mfd_1(Y,H) \ge (1-\epsilon')d_1\cdots d_{r-1}.
\]
Therefore by Theorem~\ref{thmA} applied to $(Y,H)$, if $d_r$ is sufficiently large, then
\[
d_1\cdots d_r \ge \irr(X) \ge d_r (1-\epsilon')d_1\cdots d_{r-1}-O(1).
\]
Again by taking $d_r$ sufficiently large, we have
\[
d_r (1-\epsilon')d_1\cdots d_{r-1}-O(1)\ge (1-\epsilon')^2 d_1\cdots d_r = (1-\epsilon) d_1\cdots d_r.\qedhere
\]
\end{proof}

Now we make these results more explicit in the case of general complete intersection surfaces in $\P^4$ of degree $(4,d)$ and $(5,d)$ for $d$ sufficiently large.

\begin{lemma}
If $Y\subset \P^4$ is a smooth hypersurface of degree $4$ or $5$ and $H$ is the hyperplane class on $Y$, then
\[
\mfd_1(Y,H) = \deg(Y)-1.
\]
Each fiber of such a minimal degree fibration spans a plane. If $Y$ is general, then any minimal degree fibration is equivalent (up to postcomposing by a birational automorphism of $\P^2$) to projection from a line.
\end{lemma}

\begin{proof}
First, $Y$ contains lines and projection from a line shows $\mfd_1(Y,H)\le \deg(Y)-1$.

Suppose $Y$ has degree 4. It is known that $Y$ is birationally superrigid (this was proved in \cite{IM71}, see \cite{Kollar19} for a modern perspective), so $Y$ not birational to a fibration in rational curves. Thus the only fibrations
\[
\psi\cl Y\dra \P^2
\]
of degree $\le 3$ have fibers whose closures are plane cubics.

Suppose $Y$ has degree 5. Then $Y$ is not uniruled, so there is no fibration in rational curves. The only nonrational curves of degree $\le 4$ are either curves of geometric genus 1 or plane quartics. By \cite[Thm.~1.8]{Grassi91}, $Y$ does not admit a rational fibration in genus 1 curves.

Now if $Y$ is a general quartic, the Fano curve $F(Y)$ has positive genus. As a minimal degree fibration $\psi\cl Y\dra \P^2$ is fibered in plane curves of degree $\deg(Y)-1$, this determines a rational map
\[
\P^2\dra F(Y)
\]
that sends a point $b\in \P^2$ to the residual line. This map must be constant, which proves the last part for quartics. If $Y$ is a general quintic, then $F(Y)$ is finite, so similarly the map must be constant.
\end{proof}

\begin{corollary}
Let $Y\subset \P^4$ be a general smooth hypersurface of degree $a=4$ or $5$. Let $|H|$ be the hyperplane linear series on $Y$ and let $X\in |dH|$ be a smooth surface. If $d\ge 2a^2-5a+7$ then
\[
\irr(X) = \left\{\begin{array}{rl} (d-1)(a-1)  &  \text{if $X$ contains a line},\\ d(a-1) &\text{otherwise.}  \end{array}\right.
\]
Moreover, these degree of irrationality maps factor through the minimal degree fibrations described in the previous Lemma.
\end{corollary}

\begin{proof}
With the notation from $\S3$, by the previous Lemma we have $e=a-1$, $\zeta = 5-a$, $\ell = 0$, and $M=a$. Picoco has proved (\cite[Thm. A]{Picoco}) that the answer to Question~\ref{PicocoQ} is yes for degrees up to $4$ so by Remark~\ref{PicocoQRemark} (after a computation) when $d\ge 2a^2-5a+7$, any map computing $\irr(X)$ factors through a minimal degree fibration of $Y$. As any minimal degree fibration of $Y$ is given by projection from a line in $Y$ (up to postcomposing by a Cremona transformation of the plane), the result follows from a computation.
\end{proof}

\bibliographystyle{siam} 
\bibliography{Biblio}

\end{document}